\documentclass[a4paper,12pt,leqno]{amsart}
\usepackage{amsthm,amsmath,amssymb,amscd,epic}%german,psfrag,
\usepackage[curve,matrix,arrow,frame,tips]{xy}

\usepackage{verbatim}
\usepackage{graphicx}

\DeclareMathOperator{\codim}{codim}

\DeclareMathOperator{\Gal}{Gal}

\DeclareMathOperator{\gr}{gr}

\DeclareMathOperator{\Int}{Int}

\newcommand{\mG}{{\mathbb G}}

\newcommand{\mN}{{\mathbb N}}

\newcommand{\mP}{{\mathbb P}}
\newcommand{\mQ}{{\mathbb Q}}

\newcommand{\mZ}{{\mathbb Z}}

\newcommand{\F}{{\mathcal F}}

\renewcommand{\L}{{\mathcal L}}

\renewcommand{\max}{{\rm max}}

\newcommand{\N}{{\mathcal N}}

\newcommand{\Fb}{\boldsymbol{\mathcal F}}

\renewcommand{\to}{\longrightarrow}

\newcommand{\xto}{\xrightarrow}

\newcommand{\spa}{\ }

\theoremstyle{plain}
\newtheorem{thm}{Theorem}[section]

\newtheorem{prop}[thm]{Proposition}

\theoremstyle{definition}

\newtheorem{eg}[thm]{Example}

\numberwithin{equation}{section}

\title{The fundamental group of period domains over finite fields}
\author{Sascha Orlik}
\date{}
\address{ Universit\"at Leipzig\\
Fakult\"at f\"ur Mathematik und Informatik\\
Postfach 100920\\
 04009 Leipzig\\ Germany.}
\email{orlik@mathematik.uni-leipzig.de}
\begin{document}

\maketitle

\begin{abstract}
We determine the fundamental group of period domains over finite fields. This answers a question of M. Rapoport raised in \cite{R}. 
\end{abstract}

\section{Introduction}
Period domains over finite fields are open subvarieties of flag varieties  defined by a semi-stability condition. They were introduced and discussed by M. Rapoport in \cite{R}. 
In this paper we determine their fundamental groups  which answers a question  raised in loc.cit. 

Let $G$ be a reductive group over a finite field $k$. We fix an algebraic closure
$\overline{k}$ of $k$ and denote by $\Gamma = \Gamma_k$ the corresponding absolute
Galois group of $k$. Let $\N$ be a conjugacy class of $\mQ$-$1$-PS of $G_{\overline{k}}$.
We denote by $E=E(G,\N)$ the reflex field of the pair $(G,\N)$. This is a finite extension of $k$ which
is characterized by its Galois group $\Gamma_E = \{\sigma\in\Gamma\mid \nu\in\N\Longrightarrow\nu^\sigma\in\N\}$.
Every $\mQ$-$1$-PS $\nu$ induces via Tannaka formalism a $\mQ$-filtration $\F_\nu$ over $\bar{k}$ of the forgetful  fibre functor $\omega^G: {\rm Rep_k(G)} \rightarrow {\rm Vec}_k$ from the category of algebraic $G$-representations over $k$  into the category of $k$-vector spaces. 
Two $\mQ$-1-PS are called  par-equivalent if they define the same $\mQ$-filtration. There exists a smooth projective
variety $\boldsymbol{\F} (G, \N)$ over $E$ with
\begin{equation*}
\boldsymbol{\F} (G, \N) (\overline{k}) = \{\nu\in\N\spa\text{ {\rm modulo par-equivalence} }\}\spa .
\end{equation*}
The variety is  a generalized flag variety for
$G_E$. More precisely, by a lemma of Kottwitz \cite{K}, there is a
$\mQ$-$1$-PS $\nu\in\N$ which is defined over $E = E (G, \N)$.
Thus we my write $\boldsymbol{\F} (G, \N) =G_E/P,$ where $P=P(\nu)$ is the parabolic subgroup of $G_E$ attached to $\nu$.
Further, after fixing a maximal torus and a Borel subgroup in $G$, we may suppose  that $\nu$ is contained in the closure  $\bar{C}_\mQ$ of the corresponding rational Weyl chamber $C_\mQ$. 

A point $x\in \Fb(G,\N)(\bar{k})$ is called semi-stable if the induced 
filtration $\F_x({\rm Lie} (G)_{\bar{k}})$ on the adjoint representation ${\rm Lie} (G)_{\bar{k}}={\rm Lie} (G)\otimes_k {\bar{k}}$ of $G$ is semi-stable. The latter means that for all $k$-subspaces $U$ of ${\rm Lie} (G)$,
the following inequality is satisfied
$$\frac{1}{\dim U}\Big(\sum\nolimits_y y\cdot \dim \gr_{\F|U_{\bar{k}}}^y(U_{\bar{k}})\Big) \leq \frac{1}{\dim {\rm Lie}(G)}\Big(\sum\nolimits_y y\cdot \dim \gr_{\F}^y({\rm Lie}(G)_{\bar{k}})\Big). $$  
In \cite{DOR} it is shown that there is an open subvariety $\Fb(G,\N)^{ss}$ of $\Fb(G,\N)$ parame\-trizing all semi-stable points, i.e.
$\Fb(G,\N)(\bar{k})^{ss}=\Fb(G,\N)^{ss}(\bar{k})$.  This open subvariety $\Fb(G,\N)^{ss}$ is called the {\it period domain} to $(G,\N)$.

The most prominent example of a period domain is the Drinfeld upper half plane $\Omega_k^{(\ell+1)}=\mP_k^\ell\setminus \cup\,\mP(H)$ where $H$ runs through all $k$-rational hyperplanes of $k^{\ell+1}$. This space corresponds to the pair $(G,\N)$ where $G={\rm PGL}_{\ell+1,k}$ and
$\nu=(x_1,x_2,\ldots,x_2)\in \bar{C}_\mQ$ with $x_1>x_2$ and $x_1+\ell \cdot x_2=0.$ 
Here we identify $\bar{C}_\mQ$ as usual with $(\mQ^{\ell+1})^0_+=\{(x_1,\ldots ,x_{\ell+1})\in \mQ^{\ell+1} \mid \sum_i x_i=0,\, x_1\geq x_2 \geq \ldots \geq x_{\ell+1} \}$. 
The period domain $\Omega_k^{(\ell+1)}$ is isomorphic to
a Deligne-Lusztig variety and admits therefore interesting \'etale coverings, cf. \cite{DL}. In \cite{OR} it is shown that $\Omega_k^{(\ell+1)}$ is essentially the only period domain which is at the same time a Deligne-Lusztig variety.

Period domains only depend on their adjoint data, cf. \cite{OR}, \cite{DOR}. More precisely, let
$G_{\rm ad}$ be the adjoint group of $G$, and let $\N_{\rm ad}$ be the
induced conjugacy class of $\mQ$-$1$-PS of $G_{\rm ad}$. Then 
\begin{equation*}
\boldsymbol{\F} (G, \N) (\bar{k})^{ss}\xto{\sim}\boldsymbol{\F} (G_{\rm ad}, \N_{\rm ad}) (\bar{k})^{ss}\spa .
\end{equation*}
Also if $G$ splits into a product $G=\prod_i G$, the corresponding period domain splits into products, as well.
Thus for formulating our main result, we may assume that $G$ is $k$-simple adjoint. Hence there is an absolutely simple adjoint group $G'$ over a finite extension $k'$ of $k$ with $G={\rm Res}_{k'/k} G'.$ In this case $\N=(\N_1,\ldots,\N_t)$ is given by a tuple of conjugacy classes $\N_j$ of $\mQ$-1-PS of $G'_{\bar{k}},$ where $t=|k':k|.$  Thus $\nu$ is given by a tuple of $\mQ$-1-PS $\nu=(\nu_1,\ldots,\nu_t).$ 

Our main result is the following. Let $\ell$ be the (absolute) rank of $G'.$ We denote by $\pi_1$ the functor which associates to a variety its  geometric fundamental group.

\smallskip
\noindent {\bf Theorem 1.} {\it Let $G$ be absolutely simple adjoint over $k$. Then
 $\pi_1(\Fb(G,\N)^{ss})=\{1\}$ unless  $G={\rm PGL}_{\ell+1,k}$
and $\nu=(x_1\geq x_2 \geq \ldots \geq x_{\ell+1})\in (\mQ^{\ell+1})^0_+$ with $x_2<0$ or $x_\ell>0$.
In the latter case we have $\pi_1(\Fb(G,\N)^{ss})=\pi_1(\Omega_k^{(\ell+1)})$. 

More generally, let  $G={\rm Res}_{k'/k} G'$  be $k$-simple adjoint.
Then $\pi_1(\Fb(G,\N)^{ss})=\{1\}$ unless  $G'={\rm PGL}_{\ell+1,k'}$ and  there is a unique $1\leq j \leq t,$ 
such that the following two conditions are satisfied.  Let $\nu_j=(x^{[j]}_1\geq x^{[j]}_2 \geq \ldots \geq x^{[j]}_{\ell+1})\in (\mQ^{\ell+1})^0_+,\,j=1,\ldots,t.$ Then

\smallskip
(i)  $\nu_j$ is as in the absolutely simple case, i.e., with $x^{[j]}_2<0$ or $x^{[j]}_\ell>0.$ 

\smallskip
(ii)  $\sum_{i\neq j}x^{[i]}_{1} < - x_2^{[j]}$ if $x^{[j]}_2<0$ resp. $\sum_{i\neq j}x^{[i]}_{\ell+1}> - x_\ell^{[j]}$ if $x^{[j]}_\ell> 0.$  

\smallskip
\noindent In the latter case we have $\pi_1(\Fb(G,\N)^{ss})=\pi_1(\Omega_{k'}^{(\ell+1)}).$ }

\medskip

{\it Acknowledgements:} I thank M. Rapoport for helpful remarks on this paper.

\section{Some preparations}
In this section we recall some results concerning the relation of period domains to Geometric Invariant Theory (GIT).

Let $G$ be a reductive group over $k$ and let $\N=\{\nu\}$ be a conjugacy class of $\mQ$-1-PS of $G_{\bar{k}}$.
We abbreviate $\Fb=\boldsymbol{\F} (G, \N).$ 
We fix an {\it invariant inner product} $(\,,\,)$ on  $G$ over $k$.
Recall that this is a
positive-definite bilinear form $(\,,\,)$ on $X_\ast (T)_\mQ$ for any maximal torus $T$
of $G$ defined over $\overline{k}$. The following conditions are required:

(i) For $g\in G (\overline{k})$, the inner automorphism $\Int (g)$ induces an
isometry
\begin{equation*}
\Int (g) : X_\ast (T)_\mQ\to X_\ast (T^g)\spa ,\spa T^g = g\cdot T\cdot g^{-1} \ .
\end{equation*}

(ii) Any $\sigma\in\Gamma$ induces an isometry
\begin{equation*}
\sigma : X_\ast (T)_\mQ\to X_\ast (T^\sigma)_\mQ \ .
\end{equation*}
The choice of such an inner invariant product induces together with the standard pairing $\langle\,,\,\rangle : X_\ast(T)_\mQ \times X^\ast(T)_\mQ \rightarrow \mQ$ an identification $X_\ast(T)_\mQ \cong X^\ast(T)_\mQ$ for all maximal tori $T$ of $G$ defined over $\overline{k}$.
To the pair $(G,\N)$ there is attached an ample homogeneous $\mQ$-line bundle
$\L$ on $\boldsymbol{\F}$ given by
\begin{equation*}
\L = G {\times}^P \mG_a,_{ -\nu^\ast} \spa .
\end{equation*}
Here $\nu^\ast$ denotes the rational character of $T$ which corresponds to $\nu$ under the above identification  (it extends to a character of $P$).
The following theorem of Totaro  \cite{To} describes the semi-stable points $\Fb^{ss}$ inside $\Fb$  via GIT. Here we denote by $\mu^\L (x, \lambda)$ the slope of $x\in \Fb(\bar{k})$ with respect to the 1-PS  $\lambda$ and the ample line bundle $\L$ in the sense of  GIT, cf. \cite{MFK}.
\begin{thm}\label{HM}
Let $x\in\boldsymbol{\F} (\bar{k})$.  Then $x\in \Fb^{ss}(\bar{k})$  if and only if for all $1$-PS $\lambda$ of $G_{\rm der}$ defined over
$k$ the Hilbert-Mumford inequality holds, i.e.
\begin{equation*}
\mu^\L (x, \lambda)\geq 0\spa .
\end{equation*}
\end{thm}

Let $\Delta_k=\{\alpha_1,\ldots,\alpha_d\}$ be the set of relative simple roots with respect to a fixed maximal split torus $S\subset G$ and a Borel subgroup $B\subset G$ containing $S$. Note that $G$ is quasi-split since $k$ is a finite field. Let $T=Z(S)$ be the centralizer of $S$ which is a maximal torus over $k.$ We let $\Delta$ be the set of absolutely simple roots of $G$ with respect to $T\subset B$. Then the relative simple roots are given by 
$\Delta_k=\{ \alpha|S \mid \alpha \in \Delta, \alpha|S \neq 0 \}$, cf. \cite{Ti}. By conjugating  $\nu$ with an element of the (absolute) Weyl group $W$,  we may assume that $\nu$ is contained in the closure of the dominant Weyl chamber, i.e.,
$$\nu \in \bar{C}_\mQ=\{ \lambda \in X_\ast(T)_{\mQ}\mid \langle \lambda , \alpha \rangle \geq 0 \; \forall \alpha \in \Delta \}. $$
We denote by $(\omega_\alpha)_{\alpha \in \Delta} \subset X_\ast(T)_\mQ$ the set of  co-fundamental weights.
Recall that they are defined by $(\omega_\alpha,\beta^\vee)=\delta_{\alpha,\beta}$ for $\alpha,\beta \in \Delta.$
For $1\leq i \leq d$, let 
$$\Psi(\alpha_i)=\{\beta \in \Delta \mid \beta|S=\alpha_i\}.$$ 
We set
\begin{equation}\label{relative_cofundamental}
\omega_i = \sum\nolimits_{\beta \in \Psi(\alpha_i)} \omega_\beta.
\end{equation}
Up to multiplication by a positive scalar these are just the relative fundamental weights.
In \cite{O} we have shown\footnote{Actually, in loc.cit. we considered the dual basis of $\Delta_k$ which consists of certain positive multiples of $(\omega_i)_i.$ This does not affect the statement.} that in Theorem \ref{HM} it suffices to treat the vertices of the spherical Tits-complex \cite{CLT} defined by Curtis, Lehrer and Tits.  Thus
\begin{prop}\label{vertices} Let $x\in\Fb(\bar{k}).$ Then $x\in \Fb^{ss}(\bar{k})$
iff for all $g \in G(k)$ and for all $i$ the inequality $\mu^{\mathcal L}(x, {\rm Int}(g)\circ\omega_i) \geq 0$ is satisfied.
\end{prop}

We consider the closed complement  $Y:=\Fb \setminus \Fb^{ss}$ of $\Fb^{ss}.$
For any integer $1\leq i \leq d,$ we set
$$ Y_i(\bar{k}):=\{x \in \Fb(\bar{k}) \mid  \mu^{\mathcal L}(x,\omega_i) < 0 \}.$$
The sets $Y_i(\bar{k})$ are induced by closed subvarieties $Y_i$ of $Y$ which are defined over $E.$ 
Let $P_i=P(\omega_i)$ be the parabolic subgroup corresponding to $\omega_i.$ If $n\in \mN$ is some integer such that $n\omega_i\in X_\ast(T)$,
then 
$$P(\omega_i)(\bar{k})=\{g\in G(\bar{k})\mid \lim\nolimits_{t\rightarrow 0} {\rm Int}(n\omega_i(t))\circ g\, \mbox{ exists  in }\, G(\bar{k}) \},$$ cf. \cite{MFK}. This definition does not depend on $n$ and $P_i$ is defined over $k$ since $\omega_i\in X_\ast(S)_\mQ$. 
The natural action of $G$ on $\Fb$ restricts to an action of $P_i$ on
  $Y_i$ for every $i$.
It is a consequence of Prop. \ref{vertices} that we can write $Y$ as the  union
\begin{equation}\label{union}
 Y=\bigcup\nolimits_{i=1,\ldots,d}\bigcup\nolimits_{g \in G(k)} gY_{i}.
\end{equation}

In \cite{O} we proved that the varieties $Y_i$ are unions of Schubert cells. More precisely,  denote by $W_P\subset W$ the parabolic subgroup induced by $P.$ We identify the elements of $W^P:=W/W_P$ with representatives of shortest length in $W$.
\begin{prop} We have  
\begin{eqnarray*} Y_i &=& \bigcup_{w \in W^P \atop  (\omega_i,w\nu)> 0}
  P_iw P/P \\ &=& \bigcup_{w\in W^P \atop (\omega_i,w\nu)> 0 } BwP/P.
\end{eqnarray*}
\end{prop}
The proof follows from the identity 
$$ \mu^{\mathcal L}(pw[\nu],\omega_i)= -(\omega_i,w\nu),$$
for all $p\in P_i(\bar{k})$, $w\in W.$ Here $[\nu]$ denotes the point of $\Fb(E)$ induced by $\nu.$

We conclude by (\ref{union}) that
$$\dim Y = \max_{i=1,\ldots,d} \dim Y_i. $$
On the other hand, each subvariety $Y_i$ is a union of the Schubert cells $BwP/P$, $w\in W^P,$ with $(\omega_i,w\nu)> 0$. The dimension of $BwP/P$ is $\ell(w)$, cf. \cite{Bo}.  Thus we deduce that
\begin{equation}\label{dimY_i}
 \dim Y_i = \max\; \{\ell(w) \mid w\in W^P,\, (\omega_i,w\nu) > 0 \}.
\end{equation}

Let $w_0$ resp. $w_0^P$ be the longest element of the Weyl group $W$ resp. of $W^P.$ Then $w_0=w_0^P\cdot w_P$ where $w_P$
is the longest element in $W_P.$ In particular 
\begin{equation}\label{wnu}
w_0\nu=w_0^P \nu
\end{equation}
and
\begin{equation}\label{dim}
 \dim \Fb = \ell(w_0^P).
\end{equation}
We shall examine in the next section when it happens
that $\dim Y=\dim \Fb - 1,$ i.e., $\codim Y=1.$

\section{The proof of Theorem 1}

From now on we assume that $G$ is $k$-simple adjoint, i.e., $G= {\rm Res}_{k'/k}G'$ for some
finite extension $k'/k$ of degree $t,$ cf. \cite{Ti}. Let $\ell$ be the (absolute) rank of $G'$. We start with the case where $G$ is absolutely simple adjoint  i.e., $k'=k.$

\begin{prop}\label{prop1}
 Let $G$ be absolutely  simple adjoint over $k$. Then
 $\codim Y \geq 2$ unless $G={\rm PGL}_{\ell+1}$
and $\nu=(x_1\geq x_2 \geq \ldots \geq x_\ell  \geq x_{\ell+1})\in (\mQ^{\ell+1})^0_+$ with $x_2<0$ or $x_\ell>0$.
\end{prop}

\begin{proof}
The elements of length $\ell(w_0)-1$ in $W$ are given by the expressions $s w_0$, where $s\in W$ is a simple reflection.
We deduce from (\ref{dimY_i}) -  (\ref{dim})   that there  is some integer $1\leq i \leq d$  with ${\rm codim} Y_i=1$, if and only if there is a simple reflection $s_\beta\in W,\, \beta\in \Delta,$ with
\begin{equation}\label{inequality0}
(\omega_i,s_\beta w_0\nu)> 0 .
\end{equation}
By the equivariance of $(\,,\,)$ we get
\begin{equation}\label{equivariance}
(\omega_i,s_\beta w_0\nu)=(s_\beta\omega_i,w_0\nu).
\end{equation}

\noindent {\it $1^{st}$ case:} $G$ is split.

\smallskip
Thus we have $\Delta_k=\Delta.$ 
Further, by \cite{Bou} ch. VI, 1.10, we have\footnote{Here we make use of the identification $X_\ast(T)_\mQ=X^\ast(T)_\mQ $} 
$$s_\beta\omega_i=\left\{\begin{array}{cc}
 \omega_i & \mbox{ if } \beta\neq \alpha_i \\ \omega_i-\alpha_i & \mbox{ if }\beta = \alpha_i
\end{array} \right..
 $$
Since $w_0\nu\in - \bar{C}_\mQ$ we get $(\omega_i,w_0\nu)<0$. Thus we conclude that $\beta=\alpha_i$ is a necessary condition in order that (\ref{inequality0}) holds.
Further, in this situation we get by (\ref{equivariance}) $ (\omega_i,s_\beta w_0\nu)> 0$ if and only if
\begin{equation}\label{inequality1} 
(\omega_i,w_0\nu)> (\alpha_i,w_0\nu).
\end{equation}

We start to investigate  inequality (\ref{inequality1}) for the root system of type  $A_\ell (\ell\geq 1).$
In this case the data is given as follows:
\begin{eqnarray*}
 \alpha_i & = & \epsilon_i - \epsilon_{i+1}, \; i=1,\ldots,\ell, \\
\omega_i & = & \dfrac{1}{\ell+1}\big((\ell+1-i)^{(i)},-i^{(\ell+1-i)}\big), \; i=1,\ldots,\ell,\\
\bar{C}_\mQ & = & (\mQ^{\ell+1})^0_+.
\end{eqnarray*}
Here in the definition of $\omega_i$ the exponent $(j)$  means that we repeat the corresponding entry $j$ times.
Further, $w_0$ acts on $\mQ^{\ell+1}$ via $$w_0(x_1,x_2,\ldots, x_{\ell+1}) = (x_{\ell+1}, x_\ell,\ldots,x_1).$$ 
Let $\nu=(x_1 \geq x_2 \geq \ldots \geq x_{\ell+1})\in(\mQ^{\ell+1})^0_+.$ 
Then 
$$(\omega_i,w_0\nu)= x_{\ell +1} + \ldots + x_{\ell-i+2}$$ and
$$(\alpha_i,w_0\nu)=x_{\ell-i+2} - x_{\ell-i+1}.$$
Thus  inequality (\ref{inequality1}) is satisfied if and only if
\begin{equation}\label{inequality2}
 x_{\ell+1} + \ldots + x_{\ell-i+3} > - x_{\ell-i+1}  \; \mbox{ if }\; 1< i< \ell 
\end{equation}
resp. 
$$x_{\ell}>0 \; \mbox{ if }\; i=1 $$ 
resp.
$$x_2 < 0  \; \mbox{ if }\; i=\ell  . $$
Let $1 <i < \ell.$ Then
$$x_{1} + \ldots + x_{\ell-i} + x_{\ell-i+2}  \geq x_{\ell+1} + \ldots + x_{\ell-i+3} + x_{\ell-i+1}$$
as $x_{\ell-i+2} \geq x_{\ell-i+3},\, x_{\ell-i} \geq x_{\ell-i+1}$ and $\sum_{j=1}^{\ell-i-1} x_j \geq 0$ resp.
$\sum_{j=0}^{i-3} x_{\ell+1-j} \leq 0.$ Thus (\ref{inequality2})  cannot be satisfied if $1<i<\ell$ since the sum over all entries in $\nu$  vanishes. Hence the proof follows in the case of the root system $A_\ell(\ell \geq 1).$ 

For the other split root systems, i.e., of type $B_\ell,C_\ell,D_\ell, E_6,E_7,E_8,F_4,G_2,$ we proceed as follows. 
We write down $\nu=\sum_{i=1}^\ell n_i\omega_i$ as linear combination of the co-fundamental weights with non-negative coefficients $n_i\geq 0.$ Note that $n_i=(\nu,\alpha^\vee_i),\, i=1,\ldots,\ell.$
We get 
$$w_0\nu=-\sum\nolimits_{j=1}^\ell n_{j}\omega_{\tau(j)}.$$
where $\tau$ is the opposition involution  of $\{1,\ldots,\ell\}$, cf. \cite{Ti}.  In the case of $B_\ell,C_\ell,$ $D_\ell \mbox{($\ell$ even)}, E_7,E_8,F_4,G_2$ we have $\tau={\rm id}.$ For $D_\ell \mbox{($\ell$ odd)}$, we have $\tau=(\ell-1,\ell).$  Finally in the case $E_6$ we have $\tau = (1,6)(2,5)(3,4).$
In all cases 
$$(\omega_i,w_0\nu)= -\sum\nolimits_{j=1}^\ell n_{j}(\omega_i,\omega_{\tau(j)}).$$
and
$$(\alpha_i,w_0\nu)=-n_{\tau^{-1}(i)}\cdot \frac{1}{2}\cdot (\alpha_i,\alpha_i)$$
as $\alpha_i^\vee=\frac{2\alpha_i}{(\alpha_i,\alpha_i)}.$
Since $(\omega_i,\omega_j) \geq 0$ for all $i,j, $ cf. \cite{Bou}, ch. VI, 1.10, we get
\begin{equation}\label{nocheineUngleichung}
(\omega_i,w_0\nu) \leq  -n_{\tau^{-1}(i)}\cdot (\omega_i,\omega_i).
\end{equation}
Further one checks case by case by the explicit representation of the co-fundamental weights in loc.cit. p. 265-290,  that $$(\omega_i,\omega_i) \geq \frac{1}{2}\cdot (\alpha_i,\alpha_i) \,\mbox{ for } i=1,\ldots,\ell.$$
Hence we get by using (\ref{nocheineUngleichung})
$$(\omega_i,w_0\nu) \leq  (\alpha_i,w_0\nu).$$
Thus we deduce that the inequality (\ref{inequality1}) cannot be satisfied  for root systems different from $A_\ell.$
Let us illustrate this argument for the root system of type $G_2$. Here the data is given by
\begin{eqnarray*}
\alpha_1 & = & \epsilon_1-\epsilon_2,\; \alpha_2  =  -2\epsilon_1 + \epsilon_{2}+\epsilon_3,  \\
\omega_{1} &  = &  \epsilon_3 -\epsilon_2, \;\omega_2  =  -\epsilon_1 - \epsilon_2 +  2\epsilon_3.
\end{eqnarray*}
Let $\nu= n_1\omega_1 + n_2 \omega_2$ with $n_1,n_2 \geq 0.$ 
We get $w_0\nu=- n_1\omega_1 - n_2 \omega_2.$ 
Then 
$$(\omega_1,w_0\nu)= -n_1(\omega_1,\omega_1) - n_2 (\omega_1,\omega_2)=-2n_1-3 n_2$$
and
$$(\omega_2,w_0\nu)= -n_1(\omega_2,\omega_1) - n_2 (\omega_2,\omega_2)=-3n_1-6n_2.$$
Further, we compute 
$$(\alpha_1,w_0\nu)=-n_{1}\cdot \frac{1}{2}\cdot (\alpha_1,\alpha_1)= -n_1 $$
and 
$$(\alpha_2,w_0\nu)=-n_{2}\cdot \frac{1}{2}\cdot (\alpha_2,\alpha_2)= -3n_2 .$$
Hence
$$(\omega_1,w_0\nu)\leq -n_1(\omega_1,\omega_1)=-2 n_1  \leq (\alpha_1,w_0\nu) = -n_1$$
and
$$(\omega_2,w_0\nu)\leq -n_2(\omega_2,\omega_2)=-6 n_2  \leq (\alpha_2,w_0\nu) = -3n_2.$$

\bigskip

\noindent {\it $2^{nd}$ case:} $G$ is not split. 

\smallskip 
Recall that $\omega_i = \sum_{\beta\in \Psi(\alpha_i)}  \omega_\beta,$ cf. (\ref{relative_cofundamental}). We get
$$s_\beta\omega_i=\left\{\begin{array}{cc}
 \omega_i & \mbox{ if } \beta\not\in \Psi(\alpha_i) \\ \omega_i-\beta & \mbox{ if } \beta\in \Psi(\alpha_i) 
\end{array} \right..
 $$
Again we conclude that $\beta\in\Psi(\alpha_i)$ is a necessary condition in order that (\ref{inequality0}) holds.
Further $ (\omega_i,s_\beta w_0\nu)> 0$, if and only if
\begin{equation}\label{inequality3} 
(\omega_i,w_0\nu)> (\beta,w_0\nu). 
\end{equation}
Now we have 
\begin{equation*}\label{inequality4}
(\omega_i,w_0\nu)=\sum\nolimits_{\beta\in \Psi(\alpha_i)}(\omega_\beta,w_0\nu) \leq (\omega_\beta,w_0\nu) \; \mbox{ for all }\beta \in \Psi(\alpha_i). 
\end{equation*}
Thus by the computation in the $1^{st}$ case, we conclude that a necessary condition in order that (\ref{inequality3}) holds is that the root system of $G_{\bar{k}}$ is of type $A_\ell (\ell\geq 1).$

In this case the group $G={\rm PU}_{\ell+1}$ is the projective unitary group of (absolute) rank $\ell$ and $d=[\frac{\ell+1}{2}]$, cf. \cite{Ti}. The co-fundamental weights $(\omega_i)_i$ of ${\rm PU}_{\ell+1}$ are given as follows.
Let $\Delta=\{\beta_1=\epsilon_1-\epsilon_2,\ldots,\beta_\ell=\epsilon_{\ell}-\epsilon_{\ell+1}\}$ be the set of standard simple
roots of type $A_\ell.$ Then 
$$\omega_i=\omega_{\beta_i} + \omega_{\beta_{\ell+1-i}}, \, i=1,\ldots, d-1$$
and 
$$\omega_{d}=\left\{\begin{array}{cc}
                   \omega_{\beta_d} & \mbox{ if } \frac{\ell+1}{2}\in \mZ \\
\omega_{\beta_d}+\omega_{\beta_{d+1}} &   \mbox{ if } \frac{\ell+1}{2} \not\in \mZ                
\end{array} \right. .$$
Thus by the explicit computation in the ${\rm PGL_{\ell+1}}$-case, we see that if inequality (\ref{inequality3}) is satisfied, then we necessarily have $i=1$ and 
$\beta=\beta_1$ or $\beta=\beta_\ell.$ But we compute
$$ (\omega_1,w_0\nu)= x_{\ell+1} - x_1$$
and 
$$ (\beta_1,w_0\nu)= x_{\ell+1} - x_\ell$$
resp.
$$ (\beta_\ell,w_0\nu)= x_{2} - x_1.$$
Hence we see that inequality (\ref{inequality3}) cannot be satisfied for $G={\rm PU}_{\ell+1}$ either.
\end{proof}

Next we determine explicitely  the period domains for which the codimension of the closed complement is 1.
So by Prop. \ref{prop1} we may assume that $G={\rm PGL}_{\ell+1,k}$ and  $\nu=(x_1,x_2,\ldots,x_{\ell+1})\in (\mQ^{\ell+1})^0_+$. We rewrite $\nu$ in the shape $\nu=(y_1^{(n_1)},\ldots,y_r^{(n_r)})$ with
$y_1>y_2> \cdots >y_r$ and $n_i\geq 1,\, i=1,\ldots,r.$ 
Let $V=k^{\ell+1}$. Then $\Fb(G,\N)(\bar{k})$ is given by the set of filtrations
$$(0) \subset \F^{y_1} \subset \F^{y_2} \subset \ldots \subset \F^{y_r} =V_{\bar{k}} $$
with 
$$\dim \F^{y_i}=n_1+\cdots +n_i.$$ 
If $x_2<0$ then $n_1=1$  resp. if $x_\ell>0$ then $n_r=1 .$
In order to determine the period domain, one can replace in the definition of a semi-stable filtration the Lie Algebra ${\rm Lie}(G)$ by $V$, cf. \cite{DOR}.
Thus a point $\F^\bullet$ is semi-stable if for all $k$-subspaces $U$ of $V$
the following inequality is satisfied
$$\frac{1}{\dim U}\Big(\sum\nolimits_y y\cdot \dim \gr_{\F|U_{\bar{k}}}^y(U_{\bar{k}})\Big) \leq \frac{1}{\dim V}\Big(\sum\nolimits_y y\cdot \dim \gr_{\F}^y(V_{\bar{k}})\Big). $$  
Then one  computes easily that
$$\Fb^{ss}(\bar{k})=\{\F^\bullet \in \Fb(\bar{k}) \mid \F^{y_1} \mbox{ is not contained in any $k$-rational hyperplane} \}  $$
resp.
$$\Fb^{ss}(\bar{k})=\{\F^\bullet \in \Fb(\bar{k}) \mid \F^{y_r} \mbox{ does not contain any $k$-rational line} \} . $$
Thus the projections
$$\begin{array}{ccccccc}
 \boldsymbol{\mathcal F} & \rightarrow & \mP^\ell_k & \mbox{ resp. } & \Fb & \rightarrow & \check{\mP}_k^\ell \\\\
\F^\bullet & \mapsto &  \F^{y_1} & & \F^\bullet & \mapsto &  \F^{y_r} 
\end{array}$$
induce surjective  proper maps
\begin{equation}\label{fibration}
 \begin{array}{ccccccc}
 \boldsymbol{\mathcal F}^{ss} & \rightarrow & \Omega_k^{(\ell+1)}  & \mbox{ resp. } & \Fb^{ss} & \rightarrow &  \check{\Omega}_k^{(\ell+1)} \\
\end{array}
\end{equation}
in which the  fibres are generalized flag varieties.

\medskip
{\it Proof of Theorem 1 in the absolute simple case:}  The proof follows from Proposition \ref{prop1} and the following facts on fundamental groups of algebraic varieties. If $\codim Y\geq 2,$ then we get
$\pi_1(\Fb^{ss})=\pi_1(\Fb)=\{1\},$ since $\Fb$ is simply connected, cf.  \cite{SGA1}, ch. XI, Cor. 1.2.
If $\codim Y= 1$ we are in the situation (\ref{fibration}). Then the statement follows from \cite{SGA1} Cor. 6.11 since the fibres of the maps (\ref{fibration}) are simply connected.
Note that the fundamental groups of $\Omega_k^{(\ell+1)}$ and $\check{\Omega}_k^{(\ell+1)}$  are the same since both varieties are isomorphic. 
\qed

Now we consider the general case of an $k$-simple adjoint group $G$.

\begin{prop}\label{Propo2}
Let $G={\rm Res}_{k'/k} G'$  be $k$-simple adjoint.
Then $\codim Y \geq 2$ unless $G'={\rm PGL_{\ell+1}}$ 
and there is a unique $1\leq j \leq t$, such that the following two conditions are satisfied. Let $\nu_j=(x^{[j]}_1\geq x^{[j]}_2 \geq \ldots \geq x^{[j]}_{\ell+1})\in (\mQ^{\ell+1})^0_+,\,j=1,\ldots,t.$
Then 

\smallskip
(i)   $\nu_j$ as in the absolutely simple case, i.e., with $x^{[j]}_2<0$ or $x^{[j]}_\ell>0$.

\smallskip
(ii)  $\sum_{i\neq j}x^{[i]}_{1} < - x_2^{[j]}$ if $x^{[j]}_2<0$ resp. $\sum_{i\neq j}x^{[i]}_{\ell+1}> - x_\ell^{[j]}$ if $x^{[j]}_\ell> 0.$  
\end{prop}
\begin{proof}
We conclude  by the same argument as in the proof of Proposition \ref{prop1}, $2^{nd}$ case, that ${\rm codim} Y_i=1$  if and only if there is a simple root $\beta\in\Psi(\alpha_i)$ such that 
\begin{equation}\label{inequality5} 
(\omega_i,w_0\nu)> (\beta,w_0\nu). 
\end{equation}
Let $\Gal(k'/k)=\{\sigma^j\mid 0\leq j\leq t-1\}$ and denote by $W'$  the Weyl group of $G'$. 
Since $G={\rm Res}_{k'/k} G'$  we have $W=\prod_{j=1}^tW'$ and $w_0=(w_0',\ldots,w_0')\in W$. 
Further, the natural restriction map $\Delta'_{k'} \rightarrow \Delta_k$ is bijective where
$\Delta'_{k'}=\{\alpha_1',\ldots,\alpha_d'\}$ is the set of relative simple roots of $G'$ with respect to a maximal $k'$-split torus $S'$ such that $S(k)\subset S'(k').$  
It follows that $\omega_i=\sum_{j=0}^{t-1} \sigma^j \omega_i'.$ Here $(\omega'_i)_i\in X_\ast(S')_{\mQ}$ is defined with respect to $(\alpha_i')_i\in  X^\ast(S')_{\mQ} .$ 
Furthermore, $\Delta$ is formed by $t$ copies of the set $\Delta'$ of absolute simple roots to $G'.$ We conclude that for each $\beta\in \Psi(\alpha_i)$ there is an index $j(\beta)=j$, $1\leq j\leq t,$ with 
$$(\beta,w_0\nu)=(\beta,w'_0\nu_j).$$ 
For all other indices $h\neq j$, we
have $(\beta,w'_0\nu_h)=0$.
We compute
\begin{equation}\label{inequality6}
(\omega_i,w_0\nu)=\sum\nolimits_{j=0}^{t-1}(\sigma^j\omega'_i,w_0\nu) \leq (\sigma^j\omega'_i,w_0\nu)=(\omega_i',w'_0\nu_j). 
\end{equation}
Thus by the computation in the proof of
Proposition \ref{prop1} we conclude that a necessary condition in order that (\ref{inequality5}) holds is that $G'$ is split
and that the root system of $G'$ is of type $A_\ell (\ell\geq 1).$ 

\begin{comment}
Again we have 
\begin{equation}\label{inequality6}
(\omega_i,w_0\nu)=\sum\nolimits_{\beta\in \Psi(\alpha_i)}(\omega_\beta,w_0\nu) \leq (\omega_\beta,w_0\nu)\mbox{ for all }\beta \in \Psi(\alpha_i). 
\end{equation}
Since $G={\rm Res}_{k'/k} G'$  we conclude that $\Delta$ is formed by $t$ copies of the set $\Delta'$ of absolute simple roots to $G'.$ We conclude that for each $\beta\in \Psi(\alpha_i)$ there is an index $i(\beta)=i$, $1\leq i\leq t,$ with 
$(\omega_\beta,w_0\nu)= (\omega_\beta,w_0\nu_i)$ and $(\beta,w_0\nu)=(\beta,w_0\nu_i).$ For all other indices $j\neq i$, we
have $(\omega_\beta,w_0\nu_j)=0$ and $(\beta,w_0\nu_j)=0$.
Thus by the computation in the proof of
Proposition \ref{prop1} we conclude that a necessary condition in order that (\ref{inequality5}) holds is that $G'$ is split
and that the root system of $G'$ is of type $A_\ell (\ell\geq 1).$ 
\end{comment}

So let $G'={\rm PGL}_{\ell+1,k'}.$ Then $\Delta$ is given by the set $\{\alpha_i^{[j]} \mid 1\leq i \leq \ell, 1\leq j \leq t\}$, where
$$
\alpha^{[j]}_i  =  \epsilon^{[j]}_i - \epsilon^{[j]}_{i+1}.
$$
Here $\epsilon^{[j]}_i$ is the appropriate coordinate function on $T_{\bar{k}} \cong \prod_{j=1}^tS_{\bar{k}},$ where $S$ is the diagonal torus in ${\rm PGL}_{\ell+1,k'}.$ Furthermore, the sets $\Psi(\alpha_i)$ are given by
$$\Psi(\alpha_i)=\{\alpha_i^{[j]}\mid 1\leq j \leq t\}.$$
Let $\nu=(\nu_1, \ldots,  \nu_{t})\in \bar{C}_\mQ.$ We get $w_0\nu=(w'_0\nu_1,\ldots,w'_0\nu_t)$, where the entries are given by $w'_0\nu_j=(x^{[j]}_{\ell+1},x^{[j]}_{\ell},\ldots,x^{[j]}_1),$ $j=1,\ldots,t.$
In the proof of Proposition \ref{prop1} we have seen that  if the inequalities (\ref{inequality5}) and (\ref{inequality6})  are satisfied then  $\beta=\alpha_1^{[j]}$ and $x^{[j]}_\ell>0$ resp. $\beta=\alpha_\ell^{[j]}$ and $x^{[j]}_2<0$ for some integer $j$ with $1\leq j \leq t.$

Let $\beta = \alpha^{[j]}_1$ and $x^{[j]}_\ell>0$.
Then 
$$(\omega_1,w_0\nu)= \sum_{i=1}^t x^{[i]}_{\ell +1} $$ and 
$$(\beta,w_0\nu)=x^{[j]}_{\ell+1} - x^{[j]}_{\ell}.$$
Thus the inequality (\ref{inequality5}) is satisfied if and only if
$$\sum\nolimits_{i\neq j} x^{[i]}_{\ell+1}  > - x^{[j]}_{\ell}  .$$
Furthermore,  we claim that the integer $j$ is uniquely determined. 
In fact, suppose first that $h$ is another integer with  $1\leq h\leq t$  and
$$\sum\nolimits_{i\neq h} x^{[i]}_{\ell+1}  > - x^{[h]}_{\ell} .$$
Without loss of generality we may assume that $- x^{[j]}_{\ell} \leq - x^{[h]}_{\ell}.$
Then 
$$- x^{[j]}_{\ell} \leq - x^{[h]}_{\ell} < \sum\nolimits_{i\neq h} x^{[i]}_{\ell+1} \leq x^{[j]}_{\ell+1} \leq - x^{[j]}_{\ell},$$
which is a contradiction. Here the latter inequality follows from the fact that $x^{[j]}_{\ell+1}+ x^{[j]}_{\ell} \leq 0$, since $\nu_j\in (\mQ^{\ell+1})^0_+.$

If in the opposite direction $h$ is another integer with  $1\leq h\leq t$  and
$$\sum\nolimits_{i\neq h} x^{[i]}_{1}  < - x^{[h]}_{2}$$
then
$$x^{[j]}_{1}  \leq  \sum\nolimits_{i\neq h} x^{[i]}_{1} < -x^{[h]}_{2} \leq - x^{[h]}_{\ell+1} \leq -\sum\nolimits_{i\neq j} x^{[i]}_{\ell+1}< x^{[j]}_{\ell},$$
which is a contradiction, as well. 

The case $\beta = \alpha^{(j)}_\ell$ and $x^{[j]}_2<0$ behaves dually  and yields $\sum_{i\neq j}x^{[i]}_{1} < - x_2^{[j]}.$  
\end{proof}

Again we determine explicitly the period domains where the codimension of the closed complement is 1.
So let $\nu=(\nu_1,\ldots,\nu_t) \in \bar{C}_\mQ$ such that $\codim Y=1.$   After reindexing we may suppose that $\nu_1\in (\mQ^{\ell+1})^0_+$ is the vector with $\sum_{i\neq 1}x^{[i]}_{1} < - x_2^{[1]}$ or $\sum\nolimits_{i\neq 1} x^{[i]}_{\ell+1}  > - x^{[1]}_{\ell}.$
Over the algebraic closure $\bar{k}$ the flag variety $\Fb(G,\N)$ is the product
$$\Fb(G,\N)_{\bar{k}}=\prod\nolimits_{j=1}^t \Fb({\rm PGL}_{\ell+1,\bar{k}},\N_j)_{\bar{k}}, $$  
where $\N_j$ is the ${\rm PGL}_{\ell+1,\bar{k}}$-conjugacy class of $\nu_j.$
Let $\nu_1=(y_1^{(n_1)},\ldots,y_r^{(n_r)})$ with
$y_1>y_2> \cdots >y_r$ and $n_i\geq 1,\, i=1,\ldots,r.$ 
The corresponding period domain is then given by
$$\Fb(G,\N)^{ss}_{\bar{k}}=\Fb({\rm PGL}_{\ell+1,k'},\N_1)^{ss}_{\bar{k}} \times \prod\nolimits_{j\geq 2} \Fb({\rm PGL}_{\ell+1,k'},\N_j)_{\bar{k}}.$$
In the case $\sum_{i\neq 1}x^{[i]}_{1} < - x_2^{[1]},$ we have
\begin{eqnarray*}
\Fb({\rm PGL}_{\ell+1,k'},\N_1)^{ss}(\bar{k}) & = & \{\F^\bullet \in \Fb({\rm PGL}_{\ell+1,k'},\N_1)(\bar{k}) \mid \F^{y_1} \mbox{ is not contained in} \\ & & \mbox{ any $k'$-rational hyperplane}\}.
\end{eqnarray*}
For $\sum_{i\neq 1}x^{[i]}_{\ell+1} > - x_\ell^{[1]}$, we have
\begin{eqnarray*}
\Fb({\rm PGL}_{\ell+1,k'},\N_1)^{ss}(\bar{k}) & = & \{\F^\bullet \in \Fb({\rm PGL}_{\ell+1,k'},\N_1)(\bar{k}) \mid \F^{y_r} \mbox{ does not contain } \\ & & \mbox{ any $k'$-rational line }\}. 
\end{eqnarray*}

\bigskip

{\it Proof of Theorem 1 in the general case:} The proof is the same as in the absolutely  simple case and
uses Proposition \ref{Propo2}.
\qed

\smallskip
We finish  this paper by considering a non-trivial example.

\begin{eg}  Let $G={\rm Res}_{k'/k} {\rm PGL}_{2,k'}$ with $|k':k|=2.$ Then $\nu$ corresponds to a tuple $(\nu_1,\nu_2)\in (\mQ^2)^0_+\times (\mQ^2)^0_+.$ Let $\nu_1=(x_1\geq x_2)$
and $\nu_2=(y_1\geq y_2).$ Then $x_2=-x_1\leq 0$ and $y_2=- y_1 \leq 0.$ 
If $\nu_1 \neq \nu_2$ then we are automatically in situation that w.l.o.g. $-x_2 >y_1$. Note that we allow $\nu_2=(0,0)$ to be trivial. Thus $\Fb=\mP^1 \times \mP^j, j=0,1,$  depending on whether $\nu_2$ is trivial or not. Then $E=k$ and the period domain
is given by
$$\Fb^{ss}=\Omega_{k'}^2 \times \mP^j .$$
In particular, we get $\pi_1(\Fb^{ss})= \pi_1(\Omega_{k'}^2). $
In the case $\nu_1=\nu_2$ we get $E=k'$ and
$$\Fb^{ss}=\mP^1 \times \mP^1\setminus \Delta(\mP^1(k')), $$
where $\Delta:\mP^1 \hookrightarrow \mP^1 \times \mP^1$ denotes the diagonal morphism. Here we have $\pi_1(\Fb^{ss})= \{1\}. $
\end{eg}

{

\bigskip

\end{document}